\documentclass[10pt, reqno]{amsart}

\usepackage{epsfig,latexsym,amsfonts,amssymb,amsmath,amscd,graphics,epic}
\usepackage{amsfonts,amssymb,amsmath,amscd,amsthm}
\usepackage[mathscr]{eucal}
\usepackage{mathrsfs}

\usepackage{mathbbol}

\usepackage{oldgerm,units}
\usepackage{wrapfig,epsfig}

\usepackage{ifthen}
\usepackage{mathbbol}

\usepackage{amsthm}
\usepackage[mathscr]{eucal}
\usepackage{mathrsfs}
\usepackage{mathbbol}
\usepackage{oldgerm,units}
\usepackage{wrapfig}

\newtheorem{theorem}{Theorem}[section]
\newtheorem{proposition}[theorem]{Proposition}
\newtheorem{definition}[theorem]{Definition}
\newtheorem{kdefinition}[theorem]{(Key) Definition}

\newtheorem{lemma}[theorem]{Lemma}

\newtheorem{notation}[theorem]{Notation}
\newtheorem{corollary}[theorem]{Corollary}

\newtheorem{example}[theorem]{Example}
\newtheorem{remark}[theorem]{Remark}

\newtheorem*{theorem*} {Theorem}
\newtheorem*{corollary*} {Corollary}
\newtheorem*{kquestion*} {Known question}
\newtheorem*{question*} {Question}
\newtheorem*{remark*}{Remark}
\newtheorem*{example*}{Example}
\newtheorem*{notation*}{Notation}
\newtheorem{note}[theorem]{Note}

\newcommand{\Real}{\mathbb R}

\newcommand{\one}{\mathbb{1}}
\newcommand{\zero}{\mathbb{0}}

\newcommand{\chos}[2]{{#1 \choose #2}}

\newcommand{\trop}[1]{\mathcal{#1}}

\newcommand{\tB}{\trop{B}}
\newcommand{\tC}{\trop{C}}

\newcommand{\tG}{\trop{G}}
\newcommand{\tH}{\trop{H}}

\newcommand{\tQ}{\trop{Q}}

\newcommand{\tW}{\trop{W}}

\newcommand{\al}{\alpha}

\newcommand{\Cl}[1]{#1^\bullet}

    \ifx\proof\undefined
    \newenvironment{proof}{
    \smallskip
    \noindent\emph{Proof.}}{\hfill\(\Box\)
    \bigskip
    } \fi

\newcommand{\ifdef}[3]{\ifthenelse{\equal{#1}{true}}{#2}{#3}}

\numberwithin{equation}{section}

\input xy
\xyoption{all}

\usepackage{rotating}
\usepackage{amssymb,  mathrsfs, color,  rotating}

\input amssymb.sty
\newcommand{\ds}[1]{\ {#1} \ }

\newcommand{\dss}[1]{\quad {#1} \quad }
\newcommand\nk[1]{\operatorname{c-rk}(#1)}
\newcommand\rnk[1]{\operatorname{rk}(#1)}

\def\vrp{\varphi}
\def\tlm{\widetilde{m}}

\def\htE{\widehat{E}}
\def\htB{{\widehat{B}}}
\def\htW{{\widehat{W}}}
\def\htV{{\widehat{V}}}
\def\brW{\overline{W}}
\def\htX{\widehat{X}}

\def\lat{\mathscr L}
\def\flt{F}

\def\nook{{c-rank}}
\def\nookdep{{c-dependent}}
\def\nookdepc{{c-dependence}}
\def\nookind{{c-independent}}
\def\nookindc{{c-independence}}
\def\mfl{matroid flat-lattice}

\def\tlE{\widetilde{E}}

\newcommand\Atom[1]{\operatorname{Atom}({#1})}

\newcommand\Lat[1]{\operatorname{Lat}({#1})}

\newcommand\clos{{\operatorname{cl}}}
\newcommand\hgt[1]{{\operatorname{ht}}({#1})}

\def\prll{\, {\|} \, }

\newcommand\cmp[1]{{#1}^{\operatorname{c}}}
\newcommand\trn[1]{{#1}^{\operatorname{t}}}
\newcommand\stc[1]{{#1}_{\operatorname{stc}}}

\newcommand\cl[2]{\rwcl{#1}{\str}{#2}}
\newcommand\rw[2]{\rwcl{#1}{#2}{\str}}
\newcommand\clrw[3]{#1[#3,#2]}
\newcommand\rwcl[3]{#1[#2,#3]}

\def\one{\mathbb 1} \def\zero{\mathbb 0}

\def\str{\, \ast \,}

\newcommand\thmref[2]{\pSkip\textbf{Theorem #1. }\emph{#2}\pSkip}

\def\pSkip{{\vskip 1.5mm}\noindent}
\def\ltw{0.7\textwidth}
\newcommand\HH{\mathscr{H}}
\newcommand\MM{\mathscr{M}}

\def\semiring{semiring}

\newcommand\boxtext[1]{\pSkip \qquad \qquad \qquad \framebox{\parbox{\ltw}{#1}}\pSkip}

\def\Pow{\operatorname{Pw}}
\def\iff{\Leftrightarrow}
\def\imp{\Rightarrow}
\def\bF{\mathbb F}

\def\f2A{A_{\bF_2}}

\def\H{\HH}
\def\M{\MM}

\def\Cl{\operatorname{Col}}
\def\Rw{\operatorname{Row}}

\def\sm{\setminus}

\def\1{1^\nu}

\newcommand{\etype}[1]{\renewcommand{\labelenumi}{(#1{enumi})}}
\def\eroman{\etype{\roman}}
\def\ealph{\etype{\alph}}

\def\({\left(}
\def\){\right)}

\def\tGz{{\tG_0}}

\def\tTzB{\{0,1\}}

\def\bool{\mathbb B}

\def\sbool{{\mathbb{SB}}}

\newcommand{\per}[1]{\operatorname{per}({#1})}

\begin{document}


\title[Boolean Representations of Matroids and Lattices]
{Boolean Representations of Matroids and Lattices}


\author{Zur Izhakian}

\address{School of Mathematical Sciences, Tel Aviv
     University, Ramat Aviv,  Tel Aviv 69978, Israel.
\vskip 1pt
     Department of Mathematics, Bar-Ilan University, Ramat-Gan 52900,
Israel.
    }
    \email{zzur@math.biu.ac.il}

\thanks{The research of the first author has  been  supported  by the
Israel Science Foundation (ISF grant No.  448/09) and  by the
Oberwolfach Leibniz Fellows Programme (OWLF), Mathematisches
Forschungsinstitut Oberwolfach, Germany.}

\thanks{The second author gratefully acknowledges the hospitality of the
{Mathematisches Forschungsinstut Oberwolfach} during a visit to
Oberwolfach.}

\author{John Rhodes}
\address{Department of Mathematics, University of California, Berkeley,
970 Evans Hall \#3840, Berkeley, CA 94720-3840 USA.}
\email{blvdbastille@aol.com;rhodes@math.berkeley.edu}

\thanks{\textbf{Acknowledgement:} The authors would like to thank
Professors James Oxley and Ben Steinberg
 for the useful private communications.}



\date{\today }




\begin{abstract} We introduce a new representation concept for
lattices by boolean matrices, and utilize it to prove that any
matroid is boolean representable. We show that such a
representation can be easily extracted  from a representation of
the associated lattice of flats of the matroid, leading also to a
tighter bound on the representation's size. Consequently, we
obtain a linkage of boolean representations with  geometry in a
very natural way.
\end{abstract}

\maketitle



\section*{Introduction}

A matroid is a combinatorial structure that generalizes the
familiar notion of independence in classical linear algebra; this
structure arises often in many branches of pure and applied
studies \cite{murota,Oxley03whatis,White2}. The classical theory
of matroid representations essentially deals with the realization
of matroids as ``vector spaces", allowing therefore the
utilization  both of algebraic and geometric tools in matroid
theory; matroids that do have such a realization are termed
representable. See \cite{oxley:matroid,White2}.

Over the years much effort in the study of this classical
representation theory has been invested in the attempt to specify
classes of matroids that are realizable as vector spaces defined
over fields \cite{Whittle}. The understanding that not all
matroids are representable in the field sense (see
\cite{Oxley03whatis}) has provided the motivation for considering
``vector spaces" built over other ``weaker" ground structures
instead of a field, for example a partial field
\cite{SempleWhittle} or a quasi field~\cite{DressWenzel}. Yet, the
representations taking place over these structures provide an
incomplete result -- they do not capture all matroids. As shown in
\cite{IJmat}, the superboolean semiring provides a complete
appropriate alternative framework to that of fields  which have
 customarily served for matroid representations
\cite{tutte2,whitney,Whittle}.


Superboolean representations, and more generally supertropical
representations, of hereditary collections  have been studied
initially
 in \cite{IJmat} where it was shown that every hereditary collection has a superboolean
representation and hence every matroid does also.
(Hereditary collections, also known as finite
abstract simplicial complexes, are a  much wider class of objects
which includes matroids.)

 The present paper
stresses further the study  of boolean representations, focusing
on matroids and their associated lattices; these representations
are much simpler and more accessible for computations. Our main
result is:
\thmref{\ref{cor:boolRep}}{Any matroid is boolean-representable.}
The proof of the theorem makes crucial use of the fundamental
connection, which is at the heart of our new approach, between
lattices and boolean representations, leading to a new
representation concept of finite lattices (Definitions~
\ref{kdef:latrep} and \ref{kdef:crank}). Incorporating this
lattice representation perspective, applied to the geometric
lattice of flats of a matroid, we obtain a boolean representation
for any matroid (Theorem~\ref{thm:matRep}). In particular, we
provide an explicit way to directly extract such matroid
representations from boolean representations of the associated
lattices of flats of matroids (Theorem~\ref{thm:matRep}).

The novel concept of lattice representations by boolean matrices
is presented in \S\ref{ssec:LatRep},  leading naturally to the new
notions of \nookindc \ and \nook \ of lattices (Definition
\ref{kdef:crank}). These notions are also applicable to other
abstract structures such as semilattices and partial ordered sets,
as studied in detail in \cite{IJdim,RS}.

We show that our new notion of \nookindc \ for lattices, yielding
also the notion of \nook, is properly compatible with the length
of chains of a lattice. As a consequence,  the \nook \
 of the representation of a lattice
equals the height of the lattice (Theorem~\ref{thm:htnook}). This
result has a deeper meaning, it establishes a significant
correspondence between \nookdepc \ of sup-generating subsets of a
lattice and its partitions (\S \ref{ssec:partition}, also see
Theorem \ref{thm:crossInd}). These correspondences provide a
strong evidence that our new notions are the appropriate ones for
the working  with lattices, and has strong connections with
chamber systems

The way of representing matroids via their flat-lattice
establishes an easy systemic construction procedure
Moreover, it gives us a tighter upper  bound on the representation
size of matroids (cf.~ \S\ref{ssec:bound}). Employing a natural
embedding of the boolean semiring in the tropical semiring, our
results are easily generalized further, showing that all matroids
are tropically representable as well
(Corollary~\ref{cor:tropRep}). Furthermore, more generally,  this
result also holds for any idempotent semiring.

The extra benefit arises  from our development is the important
linkage between boolean matrices and geometry, established
naturally by use of geometric lattices which are lattices of flats
of matroids \cite[Theorem 1.7.5]{oxley:matroid}.

Some indication on the connections of our results with previous
research (up to the authors limited knowledge) is in order. The
first author in \cite{IzhakianTropicalRank} in 2006 defined the
notion of independence for columns (rows) of a matrix with
coefficients in a supertropical semiring. Restricting this notion
to the superboolean semiring $\sbool$ and further to the subset of
boolean matrices, we obtain the notion of independence for columns
of a boolean matrix, to be used in this paper. Around mid 2008 the
second author saw how to apply \cite{IzhakianTropicalRank} in
other areas or mathematics, which later led to this paper and
\cite{IJdim,IJmat}.

Much earlier, in the 1990's, unknown by the first author, Dress
and Wenzel \cite{DW,WW} had also isolated  the superboolean
semiring as the correct semiring of coefficients for matroids, see
\cite[pp164]{WW}.
They showed that every matroid $\M := (E, \tH)$ gives rise to a
Grassmann-Pl\"{u}cker map from $E$  to $\sbool$. This
``determinant-type" map can be used to define independent sets of
a matroid, and in Whitney \cite{whitney} or in \S\ref{ssec:LatRep}
below, instead of the classical determinant. Thus, each matroid
has its own determinant-type map.

Our approach is extremely different. Although, as Dress and
Wenzel, we require the coefficients to be superboolean, we choose
one determinant-type map for all matroids, that is the permanent
of square superboolean matrices (cf. Lemma \ref{lem:2.1.f}). Then
we naturally adopt Whitney's \cite{whitney} approach, only that he
uses classical determinant while we use permanent. Note however
that permanent is not a Grassmann-Plucker function. This is more
in the flavor of \cite{Develin2003} than Dress and Wenzel.
    Also the authors independently rediscovered related
     results (cf.
     \S\ref{ssec:partition}) of Bjorner and Ziegler from 1991 \cite{BZ} on taking transversals
     on the partitions defined by maximal chains in the lattice of flats .


The following comments on axioms and representations of hereditary
collection and matroids may be helpful for the reader.

\pSkip \pSkip

\begin{tabular}{|l|l|l|}
  \hline
Axioms   &   Representations & Comments \\ \hline
                 Matroid   &

    Matrices over fields

  &                -- Strong exchange axiom \\  (Whitney \cite{whitney}) & &
                             -- Weak representations \\  & &
                             --     Representation size equals the matroid  rank

                                    \\ & &
                              --       Not all matroids are
                                    representable

\\ \hline
                      Hereditary collection &
                              Superboolean matrices  &
                    -- Weak axioms \\ (Def. \ref{def:hereditary}) & (See \cite{IJmat})&
                                       -- Very strong
                                       representations
\\ & &
                                      --  Every hereditary   collection is  representable                                       \\ & &
                                      --  Applications of this result are
                                      needed
\\ \hline
                           Point replacement  &
                                 Boolean matrices &

                   --   All matroids are representable (strong)
\\ (PR, Def. \ref{def:ratroid}) & &
                   -- Interesting to determine which
hereditary \\ & &  \ \ collections
 have boolean representations \\ & &  \ \ to strengthen  the PR axiom.
\\ & &
                                      -- Boolean matrices representation are \\ & & \ \ probably  not
                                      that    far from
                                       matroid.
\\ \hline
\end{tabular} \\

\begin{notation*}\label{nott:set}
 In this paper, for simplicity, we use the
following notation: Given a subset $X \subseteq E$, and elements
$x \in X$ and $p \in E$, we write $X- x$ and $X +y$ for $X \sm \{
x\}$ and $X \cup \{ y \}$, respectively; accordingly we write
$X-x+y$ for
 $(X\sm \{x\}) \cup \{ y\}$.
%
\end{notation*}

\section{Preliminaries}

\subsection{Boolean and superboolean algebras}
The \textbf{superboolean \semiring}  $\sbool : = (\{ 1, 0, 1^\nu
\},  + , \cdot
 \;)$  is a three element supertropical \semiring \ \cite{IzhakianRowen2007SuperTropical}, a
``cover'' of the familiar boolean \semiring \ $\bool := (\{ 0,1\},
+ , \cdot \;)$, endowed with the two binary operations:
$$ \begin{array}{l|lll}
   + & 0 & 1 & 1^\nu \\ \hline
     0 & 0 & 1 & 1^\nu \\
     1 & 1 & 1^\nu & 1^\nu \\
     1^\nu &1^\nu & 1^\nu & 1^\nu \\
   \end{array} \qquad
\begin{array}{l|lll}
   \cdot & 0 & 1 & 1^\nu \\ \hline
     0 & 0 & 0 & 0 \\
     1 &  0 & 1 & 1^\nu \\
     1^\nu & 0 & 1^\nu & 1^\nu \\
   \end{array}
   $$
 addition and multiplication, respectively.
This semiring is totally  ordered by $ 1^\nu \
> \ 1 \
> \ 0 .$
Note that the boolean semiring $\bool$ is an idempotent \semiring,
while $\sbool$ is not, since $1 +1 = \1$; thus $\bool$ is
\textbf{not} a subsemiring of $\sbool$. The element $\1$ is called
the \textbf{ghost} element of $\sbool$, where $\tGz := \{0, \1\}$
is the \textbf{ghost ideal} of $\sbool$. (See \cite{IJmat} for
more details.)



Superboolean matrices are matrices with entries in ~$\sbool$,
defined in the standard way, where addition and multiplication
(respecting matrix sizes) are induced from the operations of
$\sbool$ as in the familiar matrix construction.
%
%
A typical matrix is often denoted as $A =(a_{i,j})$, and  the zero
matrix is  written as $(0)$.
A boolean matrix is a matrix with coefficients in $\tTzB$. In what
follows, these matrices are considered as superboolean matrices
with entries in the subset $\tTzB \subset \sbool$.
The reader should keep in mind that the boolean matrices are only
a subset of the superboolean matrices and \textbf{not}  a
subsemiring.

\begin{definition}\label{def:matComp}
The \textbf{complement}  $\cmp{A} := (\cmp{a}_{i,j})$ of a
superboolean matrix $A = (a_{i,j})$ is defined by the rule: $$
\cmp{a}_{i,j} = 1  \dss \Leftrightarrow  a_{i,j} = 0.$$ The
\textbf{transpose} $\trn{A} = (\trn{a}_{i,j})$ of $A$ is given by
$ \trn{a}_{i,j}   = a_{j,i} .$
\end{definition}
\noindent  Note that by this definition, we have  $ \cmp{a}_{i,j}
= \1 = {a_{i,j}}$.

\begin{proposition} Transposition and complement commute, i.e.,
$\cmp{(\trn{A})} = \trn{(\cmp {A})}$  for any $n \times n$
superboolean matrix $A$.
\end{proposition}
\begin{proof} Straightforward,  $\cmp{(\trn{a}_{i,j})} = \cmp{({a_{j,i}})} =  \trn{(\cmp{a}_{i,j})}.$
\end{proof}

We define the \textbf{permanent} of an $n \times n$ superboolean
matrix $A = (a_{i,j})$ as in  the standard way:
\begin{equation}\label{eq:det1}
\per{A} := \sum _{\pi \in S_n}a_{\pi (1),1} \cdots a_{\pi (n),n},
\end{equation}
where $S_n$ stands for the group of permutations  of
$\{1,\dots,n\}$. Accordingly, \textbf{the permanent of a boolean
matrix can be $\1$}. We say that a matrix $A$ is
\textbf{nonsingular} if $\per{A} = 1$, otherwise $A$ is said to be
\textbf{singular} \cite{zur05TropicalAlgebra}.

\begin{lemma}[{\cite[Lemma 3.2]{IJmat}}]\label{lem:2.1.f}
An  $n \times n $ matrix  is nonsingular iff by independently
permuting rows and columns it has the triangular form
\begin{equation}\label{eq:trgform}
  A' :=  \(\begin{array}{cccc}
  1 & 0 &\cdots & 0 \\
  * & \ddots &  \ddots &\vdots\\
\vdots & \ddots & 1& 0\\
  * &  \cdots &*  & 1\\
   \end{array}\),
\end{equation}
 with all diagonal entries $1$, all entries above the diagonal are
 $0$, and the entries below the diagonal belong to  $\{1, \1, 0 \}$.

\end{lemma}

Lat $A$ be an $m \times n$ superboolean matrix. We say that an $k
\times \ell$ matrix $B$, with $k \leq m$ and $\ell \leq n$,  is a
\textbf{submatrix} of~$A$ if $B$ can be obtained by deleting rows
and columns of~$A$. In particular, a \textbf{row} of a matrix $A$
is an $1 \times n$ submatrix of $A$, where a \textbf{subrow} of
$A$ is an $1 \times \ell$ submatrix of $A$, with  $\ell \leq n$.


%

The following definition is key to all that follows; it includes
the definition of when a subset of columns (rows) of a
superboolean matrix are independent.

\begin{definition}[{\cite[Definition 1.2]{IzhakianTropicalRank}}]\label{def:tropicDep}
A collection of vectors  $v_1,\dots,v_m \in \sbool^{(n)}$ is said
to be (linearly) \textbf{dependent}
 if there exist $\al_1,\dots,\al_m \in \tTzB$,  not all of them $0$,
 for which
$
   \al_1 v_1 +  \cdots + \al_m  v_m \in
  \tGz^{(n)}.
$ Otherwise the vectors are said to be \textbf{independent}.
\end{definition}

The \textbf{column rank} of a superboolean matrix $A$ is defined
to be the maximal number of independent columns of $A$. The
\textbf{row rank} is defined similarly with respect to the rows of
$A$.


\begin{theorem}[{\cite[Theorem
3.11]{IzhakianTropicalRank}}]\label{thm:rnkSing} For any
superboolean matrix $A$ the row rank and the column rank are the
same, and this rank is equal to the size of the maximal
nonsingular submatrix of $A$.
\end{theorem}

\begin{corollary}[{\cite[Corollary 3.4]{IJmat}}]\label{cor:witnes}
A subset of $k$ columns (or rows) of $A$ is independent  iff it
contains a $k \times k$ nonsingular submatrix.
\end{corollary}


%
%
%
%
%
%

In the sequel, we use the following notations for submatrices:
\begin{notation}\label{nott}
We write $\cl{A}{Y}$ for the submatrix of $A$ having the column
subset $Y \subseteq \Cl(A)$, which  sometimes is refer to as a
collection of vectors, but no confusion should arise. Similarly,
we write $\rw{A}{X}$ for the submatrix of $A$ having the row
subset $X \subseteq \Rw(A)$, also refer to as a collection of
vectors. We define $\clrw{A}{Y}{X}$ to be the submatrix of $A$
having the intersection of columns $Y$ and the row subset $X
\subseteq \Rw(A)$, often also referred to as a collection of
sub-vectors.
\end{notation}


\subsection{Hereditary collections}   We
write $|E|$ for the cardinality of a given finite ground set $E$,
and $\Pow(E)$ for the \textbf{power set} of $E$. In what follows,
unless otherwise is specified, we always assume that $|E| = n$,
and thus have $|\Pow(E)| = 2^n$. Subsets of $E$ of cardinality $k$
are termed $k$-sets, for short.

\begin{definition}\label{def:hereditary}
A \textbf{hereditary collection} (or a finite abstract simplical
complex) is a pair $\H := (E,\tH)$, with $E$ finite and collection
$\tH \subseteq \Pow(E)$, that satisfies the axioms: \boxtext{
\begin{enumerate} \eroman
    \item[HT1:] \ $\tH$ is nonempty, \pSkip

    \item[HT2:]  \ $X \subseteq Y$, $Y \in \tH \ \imp \ X \in
    \tH$.
\end{enumerate}}
\end{definition}
A subset $X \in \tH$  is said to be \textbf{independent},
otherwise $ X\notin \tH$ is called \textbf{dependent}.
 A minimal dependent subset (with respect to inclusion)  of $E$ is
called a \textbf{circuit}, the collection of all circuits of a
hereditary collection $\H$ is denoted by $\tC(\H)$.
 A maximal
independent  subset (with respect to inclusion) is called a
\textbf{basis} of a hereditary collection $\H$. The set of all
bases of $\H$ is denoted as $\tB(\H) \subseteq \tH$ and termed the
\textbf{basis set} of $\H$. The \textbf{rank} $\rnk{\H}$ of  $\H$
is defined to be the cardinality of the largest member of the
basis set $\tB(\H)$ of $\H$.

\begin{definition}\label{def:HCIso}  Hereditary collections $\H_1 = (E_1,
\tH_1)$ and $\H_2 = (E_2, \tH_2)$ are said to be
\textbf{isomorphic} if there exits a bijective map $ \varphi : E_1
\to E_2$ that respects independence; that is $$\varphi(X_1) \in
\tH_2 \ \Leftrightarrow \ X_1 \in \tH_1, \qquad  \text{for any }
X_1 \subseteq E_1.$$
\end{definition}

Given a hereditary collection, we recall the following axiom:
\begin{definition}\label{def:ratroid} We say that a hereditary collection $\H = (E,\tH)$ satisfies the \textbf{point replacement
property} iff
 \boxtext{
\begin{enumerate} \eroman
    \item[PR:]  For every  $\{ p\}  \in \tH$ and every nonempty subset $J \in \tH$ there exists
    $x \in J$ such that $J - x  + p \in \tH$.
\end{enumerate}}
\end{definition}
The existence of a boolean representation (to be defined next) of
a hereditary collection implies the point replacement property.
\begin{theorem}[{\cite[Theorem 5.3]{IJmat}}] If a hereditary collection $\H = (E,\tH)$ has a  $\bool$-representation,
then $\H$ satisfies PR.
\end{theorem}

\subsection{Representation of hereditary collection}\label{sec:21.}

 Any $m \times n$  superboolean matrix $A$ gives rise to a hereditary collection
  $\H(A)$ constructed in the
following way: we label uniquely the columns of~$A$ (realized as
vectors in $\sbool^{(m)}$)  by a set $E$, $|E| = n$, the
independent subsets $\tH := \tH(A)$ of $\H$ are then subsets of
$E$ corresponding to column subsets of~$A$ that are linearly
independent in~$\sbool^{(m)}$, cf. Definition~\ref{def:tropicDep}.
Having Corollary \ref{cor:witnes}, the independent subsets of
$\H(A)$ can be described equivalently by using nonsingular
submatrices, which we call \textbf{witnesses}:
 \boxtext{
$\begin{array}{llll}
 \text{WT:} & Y \in  \tH(A) & \iff & \exists X \subseteq
\Rw(A) \text{ with } |X| = |Y|    \\[1mm]
& & & \text{such that $\clrw{A}{Y}{X}$  is nonsingular.}
\end{array}
$}
 We call $\H(A)$  an
\textbf{$\sbool$-vector hereditary collection}, and say that it is
a $\bool$-vector hereditary collection when $A$ is a boolean
matrix, cf. \cite[Definition 4.3]{IJmat}.

A hereditary collection $\H'$  is
\textbf{superboolean-representable}, written
${\sbool}$\textbf{-representable}, if it is isomorphic
(cf.~Definition~\ref{def:HCIso}) to an $\sbool$-vector hereditary
collection $\H(A)$ for some superboolean matrix $A$, and write
$A(\H)$ for an ${\sbool}$\textbf{-representation} of $\H$. When
the matrix $A(\H)$ is boolean, we call this representation a
\textbf{boolean representation}, written
${\bool}$\textbf{-representation}, and say that $\H$ is
${\bool}$\textbf{-representable}.

\begin{theorem}[{\cite[Theorem 4.6]{IJmat}}]
Any hereditary collection is superboolean-representable.
\end{theorem}

Thus, the natural question becomes: which hereditary collections
are boolean representable?

\section{Matroids and their flat lattices} To make this paper
reasonable   self contained, we open with some classical
definitions and results about matroids and their lattice of flats,
see \cite{murota,Oxley03whatis,oxley:matroid,White2}.

\begin{definition}\label{def:matroid}
A  \textbf{matroid} $\M := (E, \tH)$ is hereditary collection
 that also satisfies the following exchange axiom:
 \boxtext{
\begin{enumerate} \eroman

    \item[MT:] If $X$ and $Y$ are in $\tH$ and $|Y| = |X| + 1$, then there exists $
    y \in Y \sm X$ such that $X +y $ is in $\tH$.
\end{enumerate}}
\end{definition}

A single element $x \in E$ that forms a circuit of $\M := (E,
\tH)$, or equivalently it belongs to no basis, is called a
\textbf{loop}. Two elements $x$ and $y$ of $E$ are said to be
\textbf{parallel}, written $x \prll y$, if the $2$-set $\{x, y \}$
is a circuit of $\M$. A matroid is called \textbf{simple} if it
has no circuits consisting of $1$ or $2$ elements, i.e.,  has no
loops and no parallel elements.
 This is equivalent to all subsets with $2$ or less elements are
 independent.

The \textbf{closure} $\clos(X)$ of a subset  $X \subseteq E$ is
the subset of $E$ containing $X$ and every element $y \in
E\setminus X$ for which  there is a circuit $C \subseteq X + y$
containing $y$. This defines a closure operator $\clos: \Pow(E)
\to \Pow(E)$ which has the \textbf{Mac Lane-–Steinitz exchange
property}:
$$ \text{For any  $x, y \in E$ and all  $Y \subseteq E,$ if $x \in
\clos(Y + y ) \setminus \clos(Y)$, then $y \in \clos(Y + x)$.}$$

A subset $X \subseteq E$ is said to be \textbf{closed}, also
called a \textbf{flat}, if $X = \clos(X)$. The closed subsets of a
matroid satisfy the following properties.

\begin{enumerate} \ealph
    \item
The whole ground  set $E$ is closed. \pSkip

\item If $X$ and $Y$ are closed, then the intersection $X \cap Y$ is closed.
\pSkip

\item If $X$ is a flat, then the flats $Y$ that \textbf{cover} $X$, i.e.,
$Y$ properly contains $X$ without any flat $Z$ between $X$ and
$Y$, partition the elements of $E \setminus X$.

\end{enumerate}

The following proposition includes all the properties of flats we
will need later in this paper.

\begin{proposition}\label{rmk:closBasis} For any matroid $\M := (E,
\tH)$ we have:
\begin{enumerate} \eroman
    \item
 $\clos(B) = E$ for any basis
$B \in \tB(\M)$ of $\M$. \pSkip



\item  If $C$ is  a circuit of a matroid, then, for all $c$ in $C$,  $c$ is
a member of  $\clos(C - c)$. \pSkip

\item  $X$ is independent in a matroid iff  $x$ is
not a member of $\clos(X - x )$ for all $x \in X$.  \pSkip

\item $Y$ is dependent in a matroid iff there exists an element $y \in Y$ such that $y$ is a member of  $\clos( Y - y)$.  \pSkip

\item If $\M :=  (E, \tH)$  and $\M' := (E, \tH')$ are matroids, and no
circuit of $\M'$ lies in $\tH$, then and only then,  $\tH$ is a
subset of $\tH'$.

\end{enumerate}
\end{proposition}

\begin{proof} $ $  \begin{enumerate} \eroman
\item[(i):] By maximality of independence, for each $x \in E \sm B$,
the subset $B +  x$ has  a circuit containing~$x$. \pSkip

\item[(ii):] Clear from the definition of the closure $\clos( X  )$.
\pSkip

\item[(iii):]$X$ independent clearly implies that  $x$ is not a member of
$\clos( X - x)$, since $X$ can not contain a circuit. Also $X$ is
dependent iff $X$ contains a circuit $C$, so choosing $c$ in $C$,
implies that $c$ is a member of $\clos( X - c )$. \pSkip

\item[(iv):] The proof is logically equivalent to that of (iii). \pSkip

\item[(v):]  The statement is logically equivalent to the definition of
circuit.
\end{enumerate}
\end{proof}

The ``smaller" flats of simple matroids are easily determined:

\begin{remark}\label{rmk:singletons} When a matroid $\M:= (E, \tH)$ is simple,
the singleton  $\{ x \}$ is a flat for every $x \in E$, while
$\emptyset$ is the smallest flat (with respect to inclusion)  of
$\M$.
\end{remark}

 The class  of all flats of a simple matroid $\M$, partially ordered by set inclusion,
 forms a \textbf{\mfl}, denoted as $\Lat{\M}$,
 %
having the \textbf{top element} $T = E$ and the \textbf{bottom
element} $B = \emptyset$. The height $\hgt{\ell}$ of a lattice
element $\ell \in \lat $ is defined to be the length of the
maximal chain from $B$ to $\ell$. A lattice element of height $1$
counting edges, i.e., it covers the bottom element, is called an
\textbf{atom}.


 A finite lattice
 $\lat $ is \textbf{semimodular} if it satisfies the following
 conditions:
 \begin{enumerate} \ealph
    \item   For every pair $\{\ell,m\}$ with $\ell < m $ all the
 chains  from $\ell$ to $m$ have the same length (called the \textbf{Jordan-Dedekind chain
condition}); \pSkip
    \item $ \hgt{\ell} +  \hgt{m} \geq  \hgt{\ell \vee m} +  \hgt{\ell \wedge m},  \ \text{ for any } \ell,m \in \lat .$
 \end{enumerate}
A \textbf{geometric lattice} is a semimodular lattice in which
every element is a join of atoms.

\begin{lemma}[{\cite[Lemma 1.7.3]{oxley:matroid}}]\label{lem:oxly}
In a \mfl \ $\Lat{\M}$, for all flats $X,Y$ of $\M$
$$ X \wedge Y = X \cap Y \qquad \text{and} \qquad X \vee Y = \clos(X \cup Y).$$
\end{lemma}

\begin{theorem}[{\cite[Theorem 1.7.5]{oxley:matroid}}]
A lattice $\lat $ is geometric iff it is the lattice of flats of a
matroid, i.e., a \mfl.
\end{theorem}

\begin{corollary}\label{cor:geomlat}
Every element of the lattice of flats of a matroid is
join-generated by atoms.
\end{corollary}

\begin{remark}\label{rmk:Atoms}
In a \mfl \ $\lat  := \Lat{\M}$, with $\M := (E,\tH)$ a simple
matroid, every element $x \in E$ is closed and thus appears as an
atom $\{ x \}$ in $\lat $. Thus, there is a one-to-one
correspondence between the elements of $E$ and the atoms of~$\lat
$.
\end{remark}

 \section{Representation of lattices and partitions}

Unless otherwise is specified,  in this paper we always assume
\textbf{all lattices are finite} lattices, but almost all the
results  generalize easily to the infinite case.

 \subsection{Lattice representation}\label{ssec:LatRep}
 Within  this part of the paper, when working with lattices, we realize a matrix as a semi-module
 (see \cite[\S8,\S9]{qtheory}, called there a boolean module),  sup-generated
by the matrix rows (or columns); therefore, as explained below,
considering lattice representations we work row-wise.

Given a finite lattice $\lat  := (L, \leq)$, where $|L| = m$, we
define the $m \times m $ boolean matrix $\stc{A}(\lat ) :=
(a_{i,j})$, which we called the \textbf{structure matrix of $\lat
$}, by the rule
\begin{equation}\label{eq:matRelation}
 a_{i,j} := \left \{
\begin{array}{ll}
  1 & \text{if }  \ell_i \ds \leq \ell_j, \\[1mm]
  0 & \text{otherwise}. \\
\end{array}
\right.
\end{equation}
Accordingly, such a structure matrix has the proprieties:
\begin{enumerate} \ealph
    \item $a_{i,i} = 1$ for every $i = 1,\dots, n$, by
    reflexivity of $\lat $; \pSkip
    \item $a_{i,j} = 1$ iff $a_{j,i} =0$  for any $i \neq j$, by antisymmetry of $\lat $.
\end{enumerate}
Clearly, using the setting \eqref{eq:matRelation}, the structure
of a lattice $\lat $ is uniquely recorded by the matrix
$\stc{A}(\lat )$ and vise versa. Therefore, we identify the
lattice $\lat $ with the structure matrix $A := \stc{A}(\lat )$.
This leads us to the next two key definitions, playing a major
role in our representation theory.


The reason that ``c''  occurs in the next definitions is basically
because boolean modules are separative (the dual space of all
sup-maps into the boolean semiring $\bool$ separates points) and
the dual space is isomorphic to the original module
 with the order reversed, see \cite[Chapter 9]{qtheory}.
 The same idea  of passing to $c$ is used in \cite{IJdim}

 Also the matrix  $\stc{A}$ is triangular
with ones on the diagonal, namely is nonsingular in our sense
($\per{\stc{A}} = \one$)  which in the field sense  is invertible
(i.e., $\det (\stc{A})$), the basis of Rota's Mobius Inversion
Theorem. Therefore,  if $\stc{A}$ was used instead of applying
$c$, all subsets of lattice elements not containing the bottom
would be independent, clearly the wrong choice.

\begin{kdefinition}\label{kdef:latrep}
 The \textbf{boolean representation} $\cmp{A} := \cmp{A}(\lat)$ of a
 finite lattice
$ \lat := (L, \leq)$ is defined as
$$ \cmp{A}(\lat ) := \cmp{(\stc{A}(\lat ))},$$
also written as $\cmp{A} := (\cmp{a}_{i,j})$, cf. Definition
\ref{def:matComp}. \end{kdefinition}

This novel construction of boolean representation of lattices
leads naturally to the following fundamental notions:

\begin{kdefinition}\label{kdef:crank}
 The \textbf{\nook} of a subset $W \subseteq L
$ is then given by
$$ \nk{W} := \rnk{\rw{\cmp{A}}{W}}, \qquad \cmp{A} := \cmp{A}(\lat),$$ where  $\rw{\cmp{A}}{W}$ stands
for the rows of $\cmp{A}$ corresponding to the subset $W$, cf.
Notations \ref{nott}. We say that a subset $W \subseteq L$ is
\textbf{\nookind} if the rows $\rw{\cmp{A}}{W}$ of the matrix
$\cmp{A}$ are independent in the sense of Definition
\ref{def:tropicDep}, that is $\nk{\rw{\cmp{A}}{W}}= |W|$;
otherwise we say that $W$ is \textbf{\nookdep}.
\end{kdefinition} It easy to verify that by this definition that \nook \ is some sort of  rank function,
which always satisfies the relation
    $$\nk{W}  \ds \leq \nk{\lat } \ds \leq |L|, $$  $\text{for every } W \subseteq L.$
Actually, what the exact axioms are for this rank function is an
important open research problem.

When a subset $W \subseteq L$  with $|W| =k$ is independent, the
rows $\rw{\cmp{A}}{W}$ of the representation $\cmp{A} : =
\cmp{A}(\lat)$ contain a $k \times k $ nonsingular submatrix
$\clrw{\cmp{A}}{U}{W}$ with $U \subseteq \lat$, where $|U| = k$
(cf. Theorem \ref{thm:rnkSing}), which we call a \textbf{witness}
of $W$ (in~$\cmp{A}$). Abusing terminology, we also say that $U$
is a witness of $W$ in $\lat $. Permuting independently the rows
and columns of a witness, it has the triangular Form
\eqref{eq:trgform}, cf. Lemma \ref{lem:2.1.f}.

\begin{note}
Although the work with matroids is performed  column-wise, when
considering lattices, in order to be compatible with the order,
recorded by structure matrix, cf. \eqref{eq:matRelation}, we have
 adopted a  row-wise approach. As will be seen later, when working
with the \mfl, this approach fits well with the column-wise
representations of  matroid.
\end{note}

Aiming to establish the correspondence between the height and the
\nook \ of lattices, we need the next lemmas.
\begin{lemma}\label{lem:chain-ind}
Any strict chain $\ell_1 <   \cdots < \ell_k$, where $B < \ell_1$,
of a lattice $\lat := (L, \leq)$ determines a \nookind \ subset $W
:= \{\ell_1, \dots, \ell_k \}$ in $\lat$.
\end{lemma}

\begin{proof}
 If $\ell_1 < \cdots < \ell_k$ is a chain in $\lat $,
then $ \ell_1, \dots, \ell_k$ are independent with witness $U :=
\{ m_1, \dots, m_k\}$, where $m_1 = B$, $m_2 = \ell_1, \dots, m_k
= \ell_{k-1}$.
\end{proof}

\begin{lemma}\label{lem:inf-chain}
A witness of a \nookind \ subset $W := \{\ell_1, \dots, \ell_k \}
\subseteq L$ of a 
lattice $\lat:= (L, \leq)$ gives
rise to a strict chain of $\lat $. Detail in the proof below.
\end{lemma}

\begin{proof}
Let $\cmp{A} := \cmp{A} (\lat )$ be the matrix representation of
$\lat $, and suppose  $\clrw{\cmp{A}}{U}{W}$, where  $ U:= \{ m_1,
\dots, m_k \}$, is a witness of $W$. Permuting independently rows
and
 columns of
$\cmp{A}$, we may assume that $\clrw{\cmp{A}}{U}{W}$ is of the
triangular form~\eqref{eq:trgform}. Then, the chain
\begin{equation}\label{eq:4.9.a}
 \tlm_1 <  \tlm_2 <  \cdots <  \tlm_k, \qquad \tlm_j = m_j \wedge \cdots \wedge
m_k,
\end{equation}
is a strict chain in $\lat $. Indeed, since $\ell_1, \dots,
\ell_{k-1} \leq m_k $, $\ell_k \not \leq m_k$, in particular  $m
_k = \tlm_k$, and  we
 inductively have:
\begin{equation}\label{eq:4.9.b}
\begin{array}{rclcrcl}
    \ell_1, \ell_2,  \dots, \ell_{k-2}  & \leq &  \tlm_{k-1}, & \quad &
    \ell_{k-1},
    \ell_k  & \not \leq &  \tlm_{k-1}, \\[1mm]
    \ell_1, \dots, \ell_{k-3}  & \leq &  \tlm_{k-2} , & \quad &
    \ell_{k-2}, \ell_{k-1},
    \ell_k  & \not \leq &  \tlm_{k-2}, \\[1mm]
     \vdots \quad & &  \ \vdots & &   \vdots \quad  & &  \ \vdots \\[1mm]
      \ell_1 & \leq &  \tlm_2,
      & \quad &
    \ell_2, \dots,  \ell_{k-1},
    \ell_k  & \not \leq  & \tlm_2 ,
\end{array}
\end{equation}
and $\ell_1 \not \leq \tlm_1 =  m_1 \wedge \cdots \wedge m_{k}$.
\end{proof}

\begin{theorem}\label{thm:htnook} $\nk{\lat } = \hgt{\lat }$ for any finite meet-closed lattice
$\lat := (L,\leq)$.
\end{theorem}
\begin{proof}
Apply Lemmas \ref{lem:chain-ind} and \ref{lem:inf-chain}
respectively to a maximal strict chain and to a  basis of $\lat $.
\end{proof}

\begin{corollary}\label{cor:rankNook} Suppose $\lat := \Lat{\M}$ is the \mfl \ of $\M:=(E, \tH)$,
then  $\rnk{\M} = \nk{\lat} = \hgt{\lat }.$
\end{corollary}
%
%
%
%

In the present paper, to simplify the exposition,  we have dealt
mainly with \mfl s, i.e., geometric lattices, which are sufficient
for the purpose of matroid representations. However, a  similar
idea of boolean lattice representations is applicable for much
more general classes of lattice such as sup-generated lattices. In
\cite{IJdim} we develop the theory of lattice representations in
more generality, as well as representations of semilattices and
partial ordered sets.

\subsection{Matroid lattices and partitions}\label{ssec:partition}
Bjorner and Ziegler in \cite{BZ} have earlier results related to
the results of this section which we obtained independently.

The notion of parallel elements of matroid introduces an
equivalence relation on the ground set $E$, and thus on the
matroid $\M := (E, \tH)$. Deleting all the loops of $\M$ and then
considering  the equivalence classes $\tlE := E/ _\|$ under the
reltaion $\prll$, we get a new matroid $\widetilde{\M}$, cf.
\cite[\S1.7]{oxley:matroid}. Thus, by passing to the equivlent
classes $\tlE$, we may assume that $\prll$ is the identity, which
implies that $\widetilde \M$ is simple. Having this perspective,
in the sequel, we always assume that \textbf{all matroids are
simple}.

Given  a \mfl \ $\lat  := \Lat{\M}$,  with $\M := (E, \tH)$ a
simple matroid, then $\lat$ is geometric and meet-closed, cf.
Lemma \ref{lem:oxly}. Recall that the elements of $\lat $ are
flats of $\M$, and thus~$\lat $ has the bottom element $B =
\emptyset$ and the top element $T = E$; for notational
convenience, we denote the these flats of $\lat$ by $F_i$ while
 the atoms of $\lat$ are sometimes denoted also as $\ell_j$. Moreover,
this lattice is join-generated by the set of atoms
$$\htE := \Atom{\M} = \{ \{ x_1 \}, \dots , \{ x_n\} \}, \qquad
x_i \in E.
$$

For  a \mfl \ $\lat := \Lat{\M}$, an edge $(\flt_{i}, \flt_{i-1})$
of $\lat$ corresponds to pair of flats of $\M$, where $\flt_i$
covers the flat $\flt_{i-1}$.
 We assign to each edge $(\flt_{i}, \flt_{i-1})$ of
$\lat $ the set theoretic difference
\begin{equation}\label{eq:Qi} Q_{i} := F_i \ds \sm F_{i-1}.
\end{equation} Then, given a \emph{maximal} (strict) chain
\begin{equation}\label{eq:chainLat} E \ds = F_k \ds >  F_{k-1} \ds
> \cdots \ds >  F_1 \ds > F_0 = \emptyset, \qquad k :=
\hgt{\lat },
\end{equation} of $\lat $, from top to
bottom in $\lat $, it is easy to see that  these subsets $Q_i$ are
disjoint and their union equals $E$.

We call the collections
$$ \tQ := Q_1, \dots, Q_k, \qquad k = \hgt{\lat },$$ the \textbf{partitions} of $E$.
Note that since $\lat  := \Lat{\M}$ is semimodular, all the
partitions of $E$ are of the same size, equals the height of $\lat
$. Abusing notation we also say that $\tQ$ is a partition of the
matroid lattice $\lat  := \Lat{\M}$, with $\M := (E,\tH)$ a simple
matroid.

\begin{definition}\label{def:pcs}
A subset $W = \{ x_1, \dots, x_t \} \subseteq E$ is a
\textbf{partial transversal} of a partition $\tQ$ iff each $ x_j
\in W$ lies in a distinct $Q_i$, i.e., $|W \cap Q_i| \leq 1$ for
each $i = 1, \dots,k.$ (In such a case, we also say that $W$ is an
\textbf{independent} set of the partition $\tQ$.) A \textbf{basis}
of  a partition $\tQ$ is a partial transversal of maximal
cardinality,  equals the height of $\lat $.
\end{definition}

A  partial transversal may have less elements than the size of the
partition. One easily sees that, by the pigeonhole principle, when
a subset has a cardinality greater than the partition size (equals
the number of blocks), then it can not be a partial transversal.

\begin{example} Let $\M := U_{3,4}$ be the uniform matroid over 4
points, then the \mfl \ $\lat  := \Lat{\M}$  of $\M$ is given by
the diagram:
$$ \footnotesize   
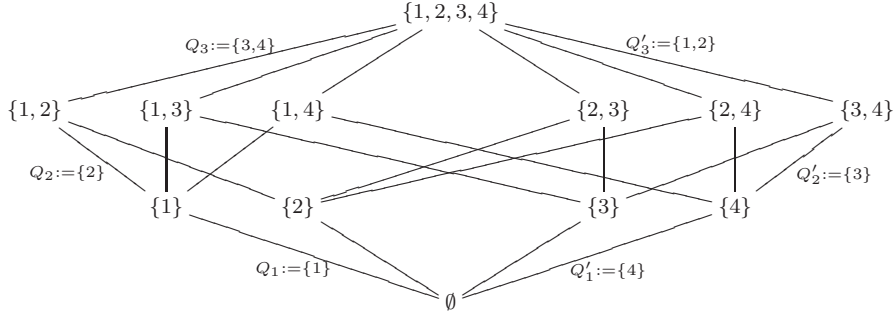
\xymatrix{
  &  & & \ar@{-}[dlll]_{Q_3 := \{ 3,4 \}} \ar@{-}[dll] \ar@{-}[dl]  \{ 1,2,3,4\}   \ar@{-}[dr]  \ar@{-}[drr] \ar@{-}[drrr]^{Q'_3 := \{ 1,2 \}}&  & \\
 \{ 1,2 \}  & \{ 1,3 \} & \{1,4 \}  & &  \{2,3 \}  & \{2,4 \}  & \{ 3,4
 \}  \\
  & \ar@{-}[ul]^{Q_2 := \{ 2 \}} \{ 1 \}  \ar@{-}[u] \ar@{-}[ur] &
\ar@{-}[ull]   \{ 2\} \ar@{-}[urr] \ar@{-}[urrr] & & \ar@{-}[ulll]
\{3\}  \ar@{-}[u] \ar@{-}[urr] &
\ar@{-}[ulll] \{ 4 \}  \ar@{-}[u] \ar@{-}[ur]_{Q'_2 := \{ 3 \}} &\\
  &  & & \ar@{-}[ull]^{Q_1 := \{ 1 \}} \ar@{-}[ul] \emptyset  \ar@{-}[ur] \ar@{-}[urr]_{Q'_1 := \{ 4 \}} &
}$$ over $12$ vertices, each corresponds to a flat of $\M$, which
has $12$ partitions.

 Two partitions of $E$, $\tQ = \{ 1\}, \{ 2\}, \{ 3,4 \}$ and
$\tQ' = \{ 1,2\}, \{3\}, \{4 \}$, are indicated on the
corresponding edges of the diagram. The maximal partial
transversals of the partition $\tQ$, i.e., the bases,  are
$\{1,2,3\}$ and $\{1,2,4\}$. It easy to see that all the bases of
the partitions are of cardinality $3$.

The representation of this \mfl \ is obtained by a $12 \times 12 $
boolean matrix.
\end{example}

\begin{remark}\label{rmk:pcs} Let  $\tQ$  be a partition of  $\lat := \Lat{\M}$.
\begin{enumerate} \eroman
    \item If $X \subseteq E$ is a partial transversal of $\tQ$,
    any subset $Y \subseteq X$ is also a partial transversal.
    \pSkip

    \item When $X \subseteq E$ is a  not partial transversal of $\tQ$,
    any subset $Z \subseteq E$  containing $X$ is not a partial transversal as well.
\end{enumerate}

\end{remark}

\begin{lemma}\label{lem:cros-ind} If $W = \{ x_1, \dots, x_t \} \subseteq E$ is a partial transversal of a partition $\tQ := Q_1, \dots , Q_k$, with $Q_i := F_i
\sm F_{i-1}$, then the corresponding atom subset $\htW = \{ \{ x_1
\} , \dots, \{ x_t \} \} \subseteq \Atom{\lat }$ is \nookind \ in
$\lat $ with witness $U \subseteq \{F_0, \dots, F_{k-1} \}$.
\end{lemma}

\begin{proof}
It is enough to prove the lemma for  $W$ a basis of the partition
$\tQ$, i.e., $t = k$ having a witness  $U =  \{F_0, \dots, F_{k-1}
\}$. Relabeling the elements of $W$, we may assume that $x_i \in
Q_i$, $i =1,\dots,k$. Let $\ell_i := \{ x_i \}$ -- the atoms of
$\lat$. Then, by construction, we have
$$
\begin{array}{rclcrcl}
    \ell_1  & \leq & F_1, F_2, \dots,  F_{k-1}, & \quad &
    \ell_1  & \not \leq &  F_0, \\[1mm]
    \ell_2  & \leq &  F_2, \dots, F_{k-1} , & \quad &
    \ell_2  & \not \leq &  F_0, F_{1}, \\[1mm]
     \vdots \  & &  \ \vdots & &   \vdots \  & &  \ \vdots \\[1mm]
      \ell_{k-1} & \leq &  F_{k-1},
      & \quad &
    \ell_{k-1}  & \not \leq  &  F_0, F_{1}, \dots, F_{k-2} ,
\end{array}
$$
and $\ell_k \not \leq F_0, F_{1}, \dots, F_{k-1}$. Writing the
matrix of these relations shows that $U$ is a witness of $W$.
\end{proof}

\begin{theorem}\label{thm:crossInd} A subset $W = \{ x_1, \dots, x_t \} \subseteq E$ is a partial transversal
 of some  partition $\tQ := Q_1, \dots , Q_k$ iff  $\htW = \{
\{ x_1 \} , \dots, \{  x_t \} \} \subseteq \Atom{\lat }$ is
\nookind \ in the \mfl \ $\lat  := \Lat{\M}$, $\M := (E , \tH)$.
\end{theorem}
\begin{proof} $(\Rightarrow):$ Immediate By Lemma \ref{lem:cros-ind}.

\pSkip  $(\Leftarrow):$ Suppose $W$ is not a partial transversal,
and let $\brW := \clos(W)$ be the closure of $W$ -- a flat
of~$\M$. Thus, $\brW$ is a proper element of the flat-lattice
$\lat$, join-generated by a subset $\htV \subseteq \Atom{\lat}$ of
atoms of $\lat$. Let $\lat'$ be the sublattice of $\lat$
consisting of all elements of $\lat$ below $\brW$, and let
$\lat'|_\htW$ be the restriction of $\lat'$ to the join-generating
subset $\htW \subseteq \htV$. Then, $\hgt{\lat'|_\htW} <  |\htW| =
|W|$, since $W$ is not a partial transversal. Thus, by Theorem
\ref{thm:htnook}, $\nk{\htW} < |\htW|$, which means that $\htW$ is
dependent in $\lat$.
\end{proof}

%

\begin{corollary} Independence of partial transversals and
the  \nookindc \ of lattice coincide.

\end{corollary}

\begin{corollary}
Maximal \nookind \ subsets of a lattice $\lat $ correspond to the
bases of its partitions.
\end{corollary}


\section{Representations of matroids} Having the method for
boolean representation of lattices at hand, together with their
connection to matroids, we can state our main result:

\begin{theorem}\label{thm:matRep}
Given a simple matroid $\M := (E, \tH)$, let $\cmp{A} :=
\cmp{A}(\lat )$ be the boolean representation of the \mfl \ $\lat
:= \Lat{\M}$, and let $\cmp{A}|_{\Atom{\lat }}$ be the restriction
of $\cmp{A}$ to the rows corresponding to the atoms of $\lat $.
Then, $\trn{(\cmp{A}|_{\Atom{\lat }})}$ is a boolean
representation of $\M$.
\end{theorem}
\begin{proof} Let $B := \{b_1, \dots, b_k \} $ be a basis of $\M$ and consider the nested  sequence of flats
$$
\begin{array}{lll}
 F_k   & := &  \clos(B) , \\[1mm]
  F_{k-1} & := &   \clos(B \sm \{ b_k \} ), \\ \ \vdots & & \quad \vdots \\[1mm]
F_{k-j} & := &  \clos(B \sm \{ b_{k-j+1}, \dots, b_k \}),
\\ \  \vdots & & \quad \vdots \\[1mm]
F_{1} & := &  \clos(B \sm \{ b_2, \dots, b_{k-1}\}), \\[1mm]
F_{0}  & :=  & \emptyset.  \\[1mm]
\end{array}
 $$
 This flat sequence introduces a  chain \begin{equation}\label{eq:chain}
\emptyset =  F_0 \ds <  F_1   \ds < \cdots \ds < F_{k-1} \ds < F_k
= E
\end{equation}  in $\lat $, where  by Proposition
\ref{rmk:closBasis}.(iii),  $\hgt{F_i} = i$ for each $i=0, \dots,
k$. (Note that we also have $k = \hgt{\lat }$.) Then, by
construction, for each $i = 1, \dots, k$
$$ b_i \in Q_i := F_{i} \sm F_{i-1},
$$
since adding $b_i$ to $F_{i-1}$ would increase the rank of
$F_{i-1}$ -- contradicting the height of $\lat $, cf. Theorem~
\ref{thm:htnook}. Thus, the chain \eqref{eq:chain} is a maximal
strict chain of $\lat $. 
Moreover,  it is also a witness for the
independence of $\htB := \{ \{ b_1\}, \dots, \{b_k\} \}$, cf.
Lemma \ref{lem:inf-chain}. By the partition method (cf.
\S\ref{ssec:partition}), $B$ is a partial transversal  of the
partition $\tQ := Q_1, \dots, Q_k$ (cf. Definition \ref{def:pcs})
and is also a basis of $\tQ.$

Since each independent subset $X \subset E$ of $\M$  is contained
in some  basis $B$, the above argument shows that $X$ is also
independent in some partition of $\lat $, and thus $\htX$ is
\nookind \ in the lattice representation $\cmp{A}(\lat )$ and in
particular in the restriction $\cmp{A}|_{\Atom{\lat}}$ to the rows
corresponding to atoms of $\lat$. \pSkip


Let  $C = \{x_1, \dots, x_t\} \subset E$ be a circuit, then  by
Proposition \ref{rmk:closBasis}.(ii), $x$ is in $\clos(C-x_i)$ for
every $x_i \in X$. We also know that $X \subseteq Y$ implies
$\clos(X) \subseteq \clos(Y)$, since $\clos$ is a closure
operator.
Assume that $X$ is a partial transversal of a partition $\tQ :=
Q_1, \dots, Q_k$, where $k \geq t$, and suppose that $x_t \in Q_k
= E \sm F_{k-1}$, i.e., $x_t$ is the closest element to the top in
$\tQ$ up to reordering. Then, there is a flat $Y$ containing $X -
x_k$ but not $x_k$, and thus $\clos(X- x_t) \subseteq \clos(Y)$.
But this is  a contradiction since $x_t \in \clos(X - x_t)$, which
is a subset of $Y$. Then,  $\htX: = \{ \{x_1\}, \dots, \{ x_t\}
\}$ is \nookdep \ in $\lat $, by Theorem \ref{thm:crossInd}. Thus,
we are done by Proposition \ref{rmk:closBasis}.(v).
\end{proof}

Composing the lattice representations of \S\ref{ssec:LatRep} with
the extraction of matroid representations as in Theorem
\ref{thm:matRep} we get the following:

\begin{theorem}\label{cor:boolRep}
Any matroid has a boolean representation.
\end{theorem}
\begin{proof} Any lattice $\lat  := (L, \leq)$, and in particular every  \mfl \ $\lat
:= \Lat{\M}$, has a boolean representation. The proof is then
completed by Theorem \ref{thm:matRep}.
\end{proof}

\subsection{Examples} We give some simple demonstrating examples
of matroid representations, extracted from their lattice
representations.

\begin{example} Let $\M$ be the simple matroid over the 5 point set $E:= \{1,2,3,4,5 \}$  whose bases are all the
3-subset except $\{1,2,3 \}$ and $\{3,4,5 \}$:
$$\footnotesize \xy 
 (18,10)*+{3},(57,2)*+{5},(57,18)*+{1}, (40,3)*+{4},
(40,17)*+{2},
(20,10)*+{\bullet},(55,4)*+{\bullet},(55,16)*+{\bullet},
(40,6)*+{\bullet}, (40,14)*+{\bullet},
(19,10)*+{}; (55,16)*+{}; **\crv{}, (19,10)*+{}; (55,4)*+{};
**\crv{},(39,15)*+{}; (61,-1)*+{};
\endxy$$

The \mfl    \ of $\lat  : = \Lat{\M}$ associated to $\M$  is then
given by the following diagram
$$ \footnotesize   
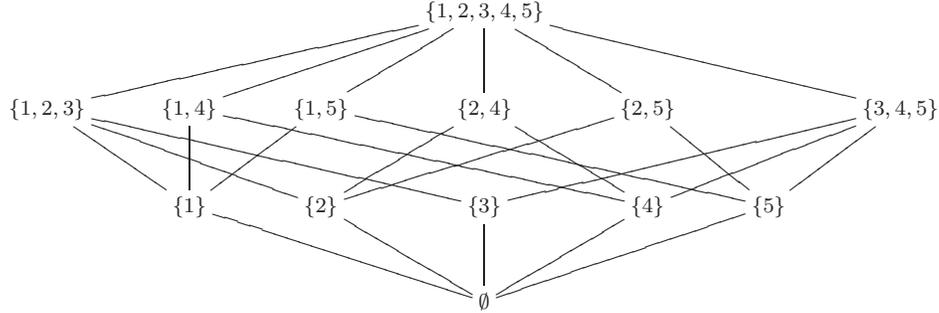
\xymatrix{
  &  & & \ar@{-}[dlll] \ar@{-}[dll] \ar@{-}[dl]  \{ 1,2,3,4,5\}  \ar@{-}[d] \ar@{-}[dr] \ar@{-}[drrr]&  & \\
 \{ 1,2,3 \}  &   \{1,4 \}  & \{1,5 \}  & \{2,4 \}  & \{ 2,5 \}
 &  & \{ 3,4,5 \} \\
  & \ar@{-}[ul] \{ 1 \}  \ar@{-}[u] \ar@{-}[ur] &
\ar@{-}[ull]   \{ 2\} \ar@{-}[ur] \ar@{-}[urr] & \ar@{-}[ulll] \{3\} \ar@{-}[urrr] &
\ar@{-}[ulll] \ar@{-}[ul] \{ 4 \} \ar@{-}[urr] & \ar@{-}[ulll] \ar@{-}[ul] \{ 5\} \ar@{-}[ur] \\
  &  & & \ar@{-}[ull] \ar@{-}[ul] \emptyset\ar@{-}[u]  \ar@{-}[ur] \ar@{-}[urr] &  &
}$$ whose 13 flats are as listed above and the atoms are the
singeltons subsets.

 The representation $\cmp{A} := \cmp{A}(\lat )$ of the \mfl \ $\lat  := \Lat{\M}$ is given by
 the following $13 \times 13 $ boolean
matrix:
$$ \cmp{A}(\lat ) \ds = \begin{array}{l||cccc cc|cc cccc cc}
        \not\leq & \emptyset  & F_1 & F_2 & F_3 & F_4  & F_5 & F_{14} & F_{15} & F_{24} & F_{25} & F_{123} &
        F_{345} & E
        \\\hline \hline
       \emptyset   & 0 & 0 & 0 & 0 & 0 & 0 & 0 & 0 & 0 & 0 & 0 & 0 & 0
       \\ \hline
        F_1 &  1 & 0 & 1 & 1 & 1 & 1 & 0 & 0 & 1 & 1 & 0 & 1 & 0 \\
        F_2 &  1 & 1 & 0 & 1 & 1 & 1 & 1 & 1 & 0 & 0 & 0 & 1 & 0 \\
        F_3 &  1 & 1 & 1 & 0 & 1 & 1 & 1 & 1 & 1 & 1 & 0 & 0 & 0 \\
        F_4 &  1 & 1 & 1 & 1 & 0 & 1 & 0 & 1 & 0 & 1 & 1 & 0 & 0 \\
     F_{5} &  1 &  1 & 1 & 1 & 1 & 0 & 1 & 0 & 1 & 0 & 1 & 0 & 0
     \\ \hline
     F_{14} &   1 & 1 & 1 & 1 & 1 & 1 & 0  & 1 & 1 & 1 & 1 & 1 & 0  \\
     F_{15} &  1 & 1 & 1 & 1 & 1 & 1 & 1  & 0 & 1 & 1 & 1 & 1 & 0\\
     F_{24} &  1 & 1 & 1 & 1 & 1 & 1 & 1  & 1 & 0 & 1 & 1 & 1 & 0\\
     F_{25} &  1 & 1 & 1 & 1 & 1 & 1 & 1  & 1 & 1 & 0 & 1 & 1 & 0\\
     F_{123} & 1 & 1 & 1 & 1 & 1 & 1 & 1  & 1 & 1 & 1 & 0 & 1 & 0 \\
     F_{345} &  1 & 1 & 1 & 1 & 1 & 1 & 1  & 1 & 1 & 1 & 1 & 0 & 0\\
     E &  1 & 1 & 1 & 1 & 1 & 1 & 1 & 1 & 1 & 1 & 1 & 1 & 0  \\
     \end{array} $$

Taking the transpose of the restriction $B := \cmp{A}|_{\Atom{\lat
}}$ of $\cmp{A}$ to the rows corresponding to the  atoms of $\lat
$, and then omitting the rows of $ \trn{B}$ corresponding  to the
atoms of $\lat $ we get a reduced representation of $\M$ by the
boolean matrix:
$$ {A}(\M) \ds = \begin{array}{l|cccc cc}
        &  F_{1} & F_{2} & F_{3} & F_{4} & F_{5} &
               \\\hline
     \emptyset & 1 & 1 &1 &1 &1 \\
     F_{14} & 0 &1  &1 & 0 & 1    \\
     F_{15} & 0 &1  &1 & 1 & 0  \\
     F_{24} & 1 & 0  &1 & 0 & 1  \\
     F_{25} & 1 & 0  &1 & 1 & 0  \\
     F_{123} & 0 & 0  &0 & 1 & 1   \\
     F_{345} & 1 &1  &0  & 0 & 0 \\
     \end{array} $$

\end{example}

\begin{example}
The matroid $K_4$ over the 6 point set $E:= \{1,2,3,4,5,6 \}$,
whose bases are all the 3-subset except $\{1,2,4 \}$, $\{1,3,5
\}$, $\{3,4,6 \}$, and $\{2,5,6\}$ corresponds to the diagram:
$$ \footnotesize 
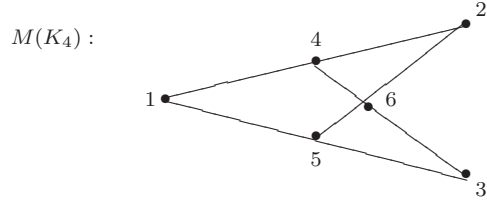
\xy 
(5,18)*+{M(K_4):}, (18,10)*+{1},(62,-2)*+{3},(62,22)*+{2},
(40,2)*+{5}, (40,18)*+{4}, (50,10)*+{6},
(20,10)*+{\bullet},(60,0)*+{\bullet},(60,20)*+{\bullet},
(40,5)*+{\bullet}, (40,15)*+{\bullet}, (47,9)*+{\bullet},
(19,10)*+{}; (61,20)*+{}; **\crv{}, (19,10)*+{}; (61,-1)*+{};
**\crv{},(39,15)*+{}; (61,-1)*+{}; **\crv{}, (39,4)*+{};
(61,21)*+{}; **\crv{},
%
\endxy$$

The matroid flat-lattice $\lat  : = \Lat{K_4}$ of $K_4$  is given
by the following diagram:

$$ \footnotesize   
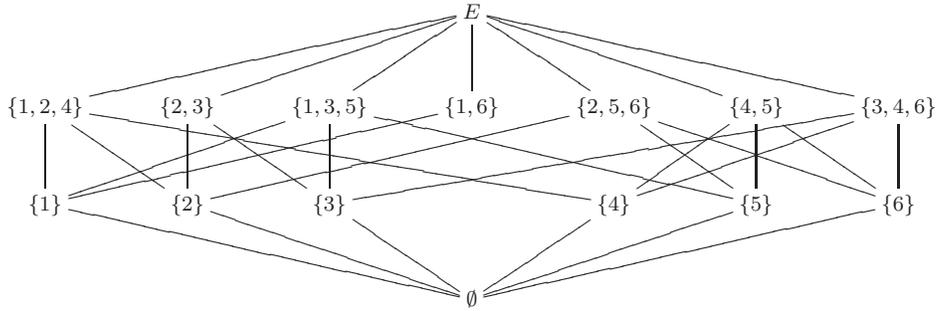
\xymatrix{
  & & & &  \ar@{-}[dlll] \ar@{-}[dll] \ar@{-}[dl]  E
    \ar@{-}[d] \ar@{-}[dr]  \ar@{-}[drr] \ar@{-}[drrr]  &  & \\
  &   \{1,2,4 \}  & \{2,3 \}  & \{1,3,5 \}  & \{ 1,6\}
 & \{2, 5,6\} & \{ 4,5 \} & \{ 3,4,6\} &  \\
  & \ar@{-}[u] \{ 1 \}  \ar@{-}[urr] \ar@{-}[urrr] &
\ar@{-}[ul] \ar@{-}[u]   \{ 2\}  \ar@{-}[urrr] & \ar@{-}[ul]
\ar@{-}[u] \{3\}
 \ar@{-}[urrrr]  & &
 \ar@{-}[ullll]   \{ 4 \} \ar@{-}[ur] \ar@{-}[urr] & \ar@{-}[ulll] \ar@{-}[ul] \{ 5\}
 \ar@{-}[u]  &   \ar@{-}[ull] \ar@{-}[ul]  \{6  \}  \ar@{-}[u]  \\
  &  & & & \ar@{-}[ulll] \ar@{-}[ull] \ar@{-}[ul] \emptyset\ar@{-}[ur]  \ar@{-}[urr] \ar@{-}[urrr] &  &
}$$ \vskip 3mm

 The representation $\cmp{A} := \cmp{A}(K_4 )$ of $\lat  $ is obtained by
 the following $15 \times 15$ boolean
matrix:
$$  \begin{array}{l|| cccc ccc|c cccc ccccccc}
        \not\leq & \emptyset  & F_1 & F_2 & F_3 & F_4  & F_5 & F_6 &
         F_{16} & F_{23} & F_{45} & F_{124} & F_{135} & F_{256} & F_{346} &
         E
        \\\hline \hline
       \emptyset   & 0 & 0 & 0 & 0 & 0 & 0 & 0 & 0 & 0 & 0 & 0 & 0 &
       0 & 0 & 0
       \\ \hline
        F_1 &  1 & 0 & 1 & 1 & 1 & 1 & 1    & 0 & 1 & 1 & 0 & 0 & 1 & 1 & 0 \\
        F_2 &  1 & 1 & 0 & 1 & 1 & 1 & 1    & 1 & 0 & 1 & 0 & 1 & 0 & 1 & 0 \\
        F_3 &  1 & 1 & 1 & 0 & 1 & 1 & 1    & 1 & 0 & 1 & 1 & 0 & 1 & 0 & 0 \\
        F_4 &  1 & 1 & 1 & 1 & 0 & 1 & 1    & 1 & 1 & 0 & 0 & 1 & 1 & 0 & 0 \\
      F_{5} &  1 & 1 & 1 & 1 & 1 & 0 & 1    & 1 & 1 & 0 & 1 & 0 & 0 & 1 & 0 \\
      F_{6} &  1 & 1 & 1 & 1 & 1 & 1 & 0    & 0 & 1 & 1 & 1 & 1 & 0 & 0 & 0
     \\ \hline
     F_{16} &  1 & 1 & 1 & 1 & 1 & 1 & 1  & 0 & 1 & 1 & 1 & 1 & 1 & 1 &  0  \\
     F_{23} &  1 & 1 & 1 & 1 & 1 & 1 & 1  & 1 & 0 & 1 & 1 & 1 & 1 & 1 &  0  \\
     F_{45} &  1 & 1 & 1 & 1 & 1 & 1 & 1  & 1 & 1 & 0 & 1 & 1 & 1 & 1 &  0  \\
     F_{124} &  1 & 1 & 1 & 1 & 1 & 1 & 1  & 1 & 1 & 1 & 0 & 1 & 1 & 1 &  0  \\
     F_{135} &  1 & 1 & 1 & 1 & 1 & 1 & 1  & 1 & 1 & 1 & 1 & 0 & 1 & 1 &  0  \\
     F_{256} &  1 & 1 & 1 & 1 & 1 & 1 & 1  & 1 & 1 & 1 & 1 & 1 & 0 & 1 &  0  \\
     F_{346} & 1 & 1 & 1 & 1 & 1 & 1 & 1  & 1 & 1 & 1 & 1 & 1 & 1 & 0 &  0  \\
     E &        1 & 1 & 1 & 1 & 1 & 1 & 1  & 1 & 1 & 1 & 1 & 1 & 1 & 1 &  0  \\
     \end{array} $$

Taking the transpose of the restriction $B := \cmp{A}|_{\Atom{\lat
}}$ of $\cmp{A}$ to the rows corresponding to the  atoms of $\lat
$, and leaving the rows of $ \trn{B}$ corresponding flats of
cardinality $\geq 2$, we get the following boolean representation
of $K_4$:
$$ {A}(K_4) \ds = \begin{array}{l|cccc ccc}
        &  F_{1} & F_{2} & F_{3} & F_{4} & F_{5} &  F_6 &
               \\\hline
     \emptyset & 1 & 1 & 1 & 1 & 1 & 1  \\
    F_{16} &       0 & 1 & 1 & 1 & 1 & 0 \\
    F_{23} &       1 & 0 & 0 & 1 & 1 & 1\\
    F_{45} &       1 & 1 & 1 & 0 & 0 & 1  \\
    F_{124}&       0 & 0 & 1 & 0 & 1 & 1   \\
    F_{135}&       1 & 1 & 0 & 1 & 0 & 1   \\
    F_{256}&       1 & 0 & 1 & 1 & 0 & 0  \\
    F_{346} &      1 & 1 & 0 & 0 & 1 & 0   \\
      \end{array} $$
\end{example}

\begin{example}
The matroid $\tW^3$ over the 6 point set $E:= \{1,2,3,4,5,6 \}$,
whose bases are all the 3-subset except $\{1,2,4 \}$, $\{1,3,5
\}$, and $\{2,3,6 \}$, has the diagram:
$$ \footnotesize 
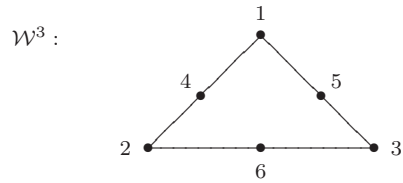
\xy  (10,15)*+{\tW^3:}, (22,0)*+{2},(58,0)*+{3},
 (40,-3)*+{6},(50,9)*+{5},
(30,9)*+{4}, (40,18)*+{1},
(25,0)*+{\bullet},(55,0)*+{\bullet},
 (40,0)*+{\bullet},(48,7)*+{\bullet},
(32,7)*+{\bullet}, (40,15)*+{\bullet},
(24,-1)*+{}; (41,16)*+{}; **\crv{}, (25,0)*+{}; (55,0)*+{};
**\crv{},(39,16)*+{}; (56,-1)*+{}; **\crv{},
%
\endxy$$

The matroid flat-lattice of $\lat  : = \Lat{\tW^3}$  is as follows
$$ \footnotesize   
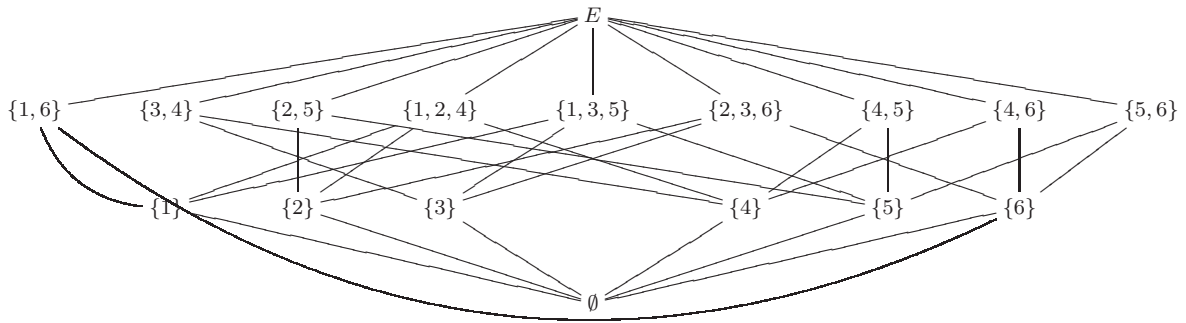
\xymatrix{
  & & & & \ar@{-}[dllll] \ar@{-}[dlll] \ar@{-}[dll] \ar@{-}[dl]  E
    \ar@{-}[d] \ar@{-}[dr]  \ar@{-}[drr] \ar@{-}[drrr]  \ar@{-}[drrrr]&  & \\
 \{ 1,6\}  &   \{3,4 \}  & \{2,5 \}  & \{1,2,4 \}  & \{ 1,3,5 \}
 & \{ 2,3,6\} & \{ 4,5 \} & \{ 4,6\} & \{5,6 \} \\
  & \ar@/^1pc/@{-}[ul] \{ 1 \}  \ar@{-}[urr] \ar@{-}[urrr] &
\ar@{-}[u]   \{ 2\} \ar@{-}[ur] \ar@{-}[urrr] & \ar@{-}[ull] \{3\}
\ar@{-}[ur] \ar@{-}[urr] & &
 \ar@{-}[ullll] \ar@{-}[ull]  \{ 4 \} \ar@{-}[ur] \ar@{-}[urr] & \ar@{-}[ullll] \ar@{-}[ull] \{ 5\}
 \ar@{-}[u] \ar@{-}[urr] &  \ar@/^5pc/@{-}[ulllllll] \ar@{-}[ull] \{6  \}  \ar@{-}[u] \ar@{-}[ur] \\
  &  & & & \ar@{-}[ulll] \ar@{-}[ull] \ar@{-}[ul] \emptyset\ar@{-}[ur]  \ar@{-}[urr] \ar@{-}[urrr] &  &
}$$ \vskip 3mm

 The following $17 \times 17$ boolean
matrix provides the  representation $\cmp{A} := \cmp{A}(\tW^3 )$
of $\lat  := \Lat{\tW^3}$:
$$  \begin{array}{l|| cccc ccc|c cccc ccccccc}
        \not\leq & \emptyset  & F_1 & F_2 & F_3 & F_4  & F_5 & F_6 &
         F_{16} & F_{25} & F_{34} & F_{45} & F_{46} & F_{56} & F_{124} &
        F_{135} & F_{236}& E
        \\\hline \hline
       \emptyset   & 0 & 0 & 0 & 0 & 0 & 0 & 0 & 0 & 0 & 0 & 0 & 0 &
       0 & 0 & 0 & 0 & 0
       \\ \hline
        F_1 &  1 & 0 & 1 & 1 & 1 & 1 & 1    & 0 & 1 & 1 & 1 & 1 & 1 & 0 & 0 & 1 & 0 \\
        F_2 &  1 & 1 & 0 & 1 & 1 & 1 & 1    & 1 & 0 & 1 & 1 & 1 & 1 & 0 & 1 & 0 & 0 \\
        F_3 &  1 & 1 & 1 & 0 & 1 & 1 & 1    & 1 & 1 & 0 & 1 & 1 & 1 & 1 & 0 & 0 & 0 \\
        F_4 &  1 & 1 & 1 & 1 & 0 & 1 & 1    & 1 & 1 & 0 & 0 & 0 & 1 & 0 & 1 & 1 & 0 \\
      F_{5} &  1 & 1 & 1 & 1 & 1 & 0 & 1    & 1 & 0 & 1 & 0 & 1 & 0 & 1 & 0 & 1 & 0 \\
      F_{6} &  1 & 1 & 1 & 1 & 1 & 1 & 0    & 0 & 1 & 1 & 1 & 0 & 0 & 1 & 1 & 0 & 0
     \\ \hline
     F_{16} &  1 & 1 & 1 & 1 & 1 & 1 & 1  & 0 & 1 & 1 & 1 & 1 & 1 & 1 & 1 & 1 &  0  \\
     F_{25} &  1 & 1 & 1 & 1 & 1 & 1 & 1  & 1 & 0 & 1 & 1 & 1 & 1 & 1 & 1 & 1 &  0  \\
     F_{34} &  1 & 1 & 1 & 1 & 1 & 1 & 1  & 1 & 1 & 0 & 1 & 1 & 1 & 1 & 1 & 1 &  0  \\
     F_{45} &  1 & 1 & 1 & 1 & 1 & 1 & 1  & 1 & 1 & 1 & 0 & 1 & 1 & 1 & 1 & 1 &  0  \\
     F_{46} &  1 & 1 & 1 & 1 & 1 & 1 & 1  & 1 & 1 & 1 & 1 & 0 & 1 & 1 & 1 & 1 &  0  \\
     F_{56} &  1 & 1 & 1 & 1 & 1 & 1 & 1  & 1 & 1 & 1 & 1 & 1 & 0 & 1 & 1 & 1 &  0  \\
     F_{124} & 1 & 1 & 1 & 1 & 1 & 1 & 1  & 1 & 1 & 1 & 1 & 1 & 1 & 0 & 1 & 1 &  0  \\
     F_{135} &  1 & 1 & 1 & 1 & 1 & 1 & 1  & 1 & 1 & 1 & 1 & 1 & 1 & 1 & 0 & 1 &  0  \\
     F_{236} &  1 & 1 & 1 & 1 & 1 & 1 & 1  & 1 & 1 & 1 & 1 & 1 & 1 & 1 & 1 & 0 &  0  \\
     E &        1 & 1 & 1 & 1 & 1 & 1 & 1  & 1 & 1 & 1 & 1 & 1 & 1 & 1 & 1 & 1 &  0  \\
     \end{array} $$

The transpose of the restriction $B := \cmp{A}|_{\Atom{\lat }}$ of
$\cmp{A}$ to the rows corresponding to the  atoms of $\lat $ gives
us the following boolean representation of $\tW^3$:
$$ {A}(\tW^3) \ds = \begin{array}{l|cccc ccc}
        &  F_{1} & F_{2} & F_{3} & F_{4} & F_{5} &  F_6 &
               \\\hline
     \emptyset & 1 & 1 & 1 & 1 & 1 & 1  \\
     F_{16} &       0 & 1 & 1 & 1 & 1 & 0 \\
     F_{25} &       1 & 0 & 1 & 1 & 0 & 1\\
     F_{34} &       1 & 1 & 0 & 0 & 1 & 1  \\
     F_{45} &       1 & 1 & 1 & 0 & 0 & 1   \\
     F_{46} &       1 & 1 & 1 & 0 & 1 & 0   \\
     F_{56} &       1 & 1 & 1 & 1 & 0 & 0  \\
     F_{124} &      0 & 0 & 1 & 0 & 1 & 1   \\
     F_{135} &      0 & 1 & 0 & 1 & 0 & 1    \\
     F_{236} &      1 & 0 & 0 & 1 & 1 & 0    \\
      \end{array} $$
(Note that this representation can be reduced further by omitting
duplicate rows.)
\end{example}

\subsection{An upper  bound on the representation size}\label{ssec:bound} Using the
new representation of matroids, assisted by Corollary
\ref{cor:rankNook}, we compute an upper bound for the size of the
boolean matroids, i.e., the height of the representing boolean
matrix. Let $A(\M) :=  \cmp{A}|_{\Atom{\lat }}$, $\lat :=
\Lat{\M}$, be a boolean representation of the matroid $\M := (E,
\tH)$, as obtained from Theorem \ref{thm:matRep}. Suppose $\M$ has
 rank $k$ and $A(\M)$ is an $m \times n$ matrix, i.e., $|E| = n$.
Then, we have the following naive upper bound:
$$ m  \ds \leq \sum_{i = 0} ^{k} \chos{n}{i}.$$

Of course, a better upper bound on the size is the number of sji
(strict join irreducibles, see  \cite[\S6]{qtheory} and
\cite{IJdim}, also see \cite{IJdim}).


\subsection{Tropical representations}\label{ssec:tropRep}
In \cite[Apendix A]{IJmat} we have shown that the boolean semiring
$\bool := (\{ 0,1\}, + , \cdot \;)$ embeds naturally in the
tropical semiring $\Real_{(\max, + )} := (\Real \cup \{ -\infty\},
\max, +)$, or dually in $\Real_{(\min, + )} := (\Real \cup \{
\infty\}, \min, +)$, and much more generally it embeds in any
idempotent semiring $S := (S,+, \cdot \;)$ by sending $1 \mapsto
\one_S$ and $0 \mapsto \zero_S$,  the multiplicative unit and the
zero of $S$, respectively. In particular, for the tropical
semiring $S = \Real_{(\max, + )}$, the embedding $\vrp : \bool
\hookrightarrow \Real_{(\max, + )}$ is given by $\vrp : 1 \mapsto
0$ and $\vrp: 0 \mapsto -\infty$.

Having this embedding $\vrp: \bool \hookrightarrow S$, we easily
generalize the result of Theorem \ref{cor:boolRep}:
\begin{corollary}\label{cor:tropRep}
Every matroid is representable over any idempotent semiring, and
in particular each matroid is tropically representable.
\end{corollary}


\end{document}